\documentclass[oumk,xs]{xxllp}
\usepackage{amsmath,amssymb, amsfonts} 

\usepackage[TS1,T1]{fontenc}
\usepackage{xxllp}
\usepackage{LLPthm}
\usepackage[mathb]{xlmodern} 

\usepackage[breaklinks,colorlinks,citecolor=blue,urlcolor=blue]{hyperref}
\newcommand\DOI[2]{DOI: \href{#1#2}{#2}}
\usepackage[sort,elide]{natbib}
\renewcommand{\pubonline}{XXXXX XX, 202X}
\newcommand*{\mies}{XXXXXX}
\renewcommand\figurename{\small Figure}
\renewcommand\tablename{\small Table}
\newtheorem{theorem}{Theorem}[section]

\newtheorem{corollary}[theorem]{Corollary}

\theoremstyle{remark}

\theoremstyle{definition}
\newtheorem{definition}{Definition}[section]

\renewcommand{\qedsymbol}{$\dashv$}

\usepackage[HTML]{xcolor}
\usepackage{nc,ifthen,adjustbox} 
\usepackage{listings}

\definecolor{nred}{HTML}{E8170C}
\definecolor{norange}{HTML}{F09F18}
\definecolor{npurple}{HTML}{8500EB}
\definecolor{ngreen}{HTML}{06C753}
\definecolor{nblue}{HTML}{049BFF}

\def\lr#1{\multicolumn{1}{|@{}c@{}}{#1}}
\def\rr#1{\multicolumn{1}{@{}c@{}|}{#1}}
\def\lrr#1{\multicolumn{1}{|@{}c@{}|}{#1}}
\def\nc#1#2{{\color{#1}\ensuremath{\,#2\,}}}
\def\dc#1#2{\colorbox{#1}{\color{white}\ensuremath{#2}}}

\def\AC{\ensuremath{\mathbf{AC}}}
\def\L{\ensuremath{\mathbf{L}}}
\def\KW{\ensuremath{\mathbf{K_w}}}
\def\NC{\ensuremath{\mathbf{NC}}}
\def\ACT{\ensuremath{{\mathbf{AC}_2}}}
\def\FDE{\ensuremath{\mathbf{FDE}}}
\def\FC{\ensuremath{\mathbf{FC}}}

\def\ncf{\mathfrak{f}}
\def\ncu{\mathfrak{u}}
\def\nct{\mathfrak{t}}
\def\acf{\mathbf{\scriptstyle f}}
\def\acb{\mathbf{\scriptstyle b}}
\def\acn{\mathbf{\scriptstyle n}}
\def\act{\mathbf{\scriptstyle t}}
\def\fdef{\mathbf{\scriptstyle F}}
\def\fdeb{\mathbf{\scriptstyle B}}
\def\fden{\mathbf{\scriptstyle N}}
\def\fdet{\mathbf{\scriptstyle T}}

\def\ncff{\nc{nred}{\ncf\ncf}}
\def\ncfu{\nc{norange}{\ncf\ncu}}
\def\ncft{\nc{ngreen}{\ncf\nct}}
\def\ncuf{\nc{npurple}{\ncu\ncf}}
\def\ncuu{\nc{darkgray}{\ncu\ncu}}
\def\ncut{\nc{nblue}{\ncu\nct}}
\def\nctf{\dc{ngreen}{\nct\ncf}}
\def\nctu{\dc{nblue}{\nct\ncu}}
\def\nctt{\dc{nred}{\nct\nct}}

\def\acNn{\nc{nred}{\fden\acn}}
\def\acFNf{\nc{norange}{\fdef\fden\acf}}
\def\acFn{\nc{ngreen}{\fdef\acn}}
\def\acNTt{\nc{npurple}{\fden\fdet\act}}
\def\acBFNTb{\nc{darkgray}{\fdeb{\cdot}\fden\acb}}
\def\acBtFt{\nc{nblue}{\fdeb\fdef\act}}
\def\acTn{\dc{ngreen}{\fdet\acn}}
\def\acBTf{\dc{nblue}{\fdeb\fdet\acf}}
\def\acBn{\dc{nred}{\fdeb\acn}}

\def\fcnn{\nc{nred}{\fden\acn}}
\def\fcff{\nc{norange}{\fdef\acf}}
\def\fcnf{\nc{norange}{\fden\acf}}
\def\fcfn{\nc{ngreen}{\fdef\acn}}
\def\fcnt{\nc{npurple}{\fden\act}}
\def\fctt{\nc{npurple}{\fdet\act}}
\def\fcfb{\nc{darkgray}{\fdef\acb}}
\def\fcnb{\nc{darkgray}{\fden\acb}}
\def\fcbb{\nc{darkgray}{\fdeb\acb}}
\def\fctb{\nc{darkgray}{\fdet\acb}}
\def\fcft{\nc{nblue}{\fdef\act}}
\def\fcbt{\nc{nblue}{\fdeb\act}}
\def\fctn{\dc{ngreen}{\fdet\acn}}
\def\fcbf{\dc{nblue}{\fdeb\acf}}
\def\fctf{\dc{nblue}{\fdet\acf}}
\def\fcbn{\dc{nred}{\fdeb\acn}}

\def\dom{\mathrm{dom}}
\def\ran{\mathrm{ran}}

\begin{document}
\LaLPonline{1} \setcounter{page}{1} \thispagestyle{LLProbezissn} \label{p}


\AuthorTitle{Richard Zach}{An Epimorphism between Fine and Ferguson's Matrices for Angell's~\AC}

\allowdisplaybreaks

\Abstract{Angell's logic of analytic containment \AC{} has been shown to be
characterized by a $9$-valued matrix \NC{} by Ferguson, and by a
$16$-valued matrix by Fine.  It is shown that the former is the image
of a surjective homomorphism from the latter, i.e., an epimorphic
image. Some candidate $7$-valued matrices are ruled out as
characteristic of~\AC. Whether matrices with fewer than $9$ values
exist remains an open question. The results were obtained with the
help of the MUltlog system for investigating finite-valued logics; the
results serve as an example of the usefulness of techniques from
computational algebra in logic. A tableau proof system for~\NC{} is
also provided.}


\section{Introduction}

\citet{Ferguson2016} and \citet{Fine2016} have independently provided
many-valued matrices which characterize the logic of analytic
containment~\AC{} of \citet{Angell1977,Angell1989}.  Ferguson's
matrix~\NC{} has nine truth values, while Fine's matrix \FC{} has 16.
\citet[fn.~1 on p.~200]{Fine2016} suggested that it would be of interest to
compare the two. We show below that \NC{} is an epimorphic image
of~\FC, i.e., \NC{} is isomorphic to a factor matrix of~\FC.

Ferguson's \NC{} is defined as follows. We start with weak Kleene
logic, which has three truth values $K = \{\ncf, \ncu, \nct\}$. The
truth tables for $\land$ and $\lor$ are familiar:
\[\begin{array}[t]{c|ccc}\\
   \land^{\mathbf{K}_w} & \ncf & \ncu & \nct\\ \hline
   \ncf & \ncf & \ncu & \ncf\\
   \ncu & \ncu & \ncu & \ncu\\
   \nct & \ncf & \ncu & \nct\\
   \end{array}\qquad
   \begin{array}[t]{c|ccc}\\
    \lor^{\mathbf{K}_w} & \ncf & \ncu & \nct\\ \hline
   \ncf & \ncf & \ncu & \nct\\
   \ncu & \ncu & \ncu & \ncu\\
   \nct & \nct & \ncu & \nct\\
   \end{array}\]
\NC{} has the truth values $\mathit{NC} = K \times K$, of which
$\mathit{NC}^+ = \{\nct\} \times K$ are designated, with the
truth functions defined by:
\begin{align*}
\lnot^\NC(\langle v_0, v_1\rangle) & = \langle v_1, v_0\rangle\\
\land^\NC(\langle v_0, v_1\rangle, \langle v_0', v_1'\rangle) & =
    \langle \land^\KW(v_0, v_0'), \land^\KW(v_1, v_1')\rangle\\
\lor^\NC(\langle v_0, v_1\rangle, \langle v_0', v_1'\rangle) & =
    \langle \lor^\KW(v_0, v_0'), \lor^\KW(v_1, v_1')\rangle
\end{align*}
That is, $\land$ and $\lor$ are defined component-wise (as in the
direct product of two matrices). However, $\lnot$ is not the
component-wise weak Kleene negation, but instead switches the truth
values in a pair.

Dunn and Belnap's matrix for  \FDE{} has the four truth values $\mathit{FDE} = \{\fdeb,
\fdet, \fdef, \fden\}$ with designated values $\mathit{FDE}^+ = \{\fdeb,
\fdet\}$ and truth functions as follows:\footnote{We use $\fdeb, \fdet,
\fdef, \fden$ instead of Fine's TF, T$\!\not\!\mathrm{F}$,
$\not\!\mathrm{T}$F, $\not\!\mathrm{T}${}$\!\not\!\mathrm{F}$ to save space.}
\[\begin{array}[t]{c|c}
\lnot^\FDE &\\
\hline
\fdeb & \fdeb\\
\fdet & \fdef\\
\fdef & \fdet\\
\fden & \fden
\end{array}\qquad
\begin{array}[t]{c|cccc}
\land^\FDE & \fdeb & \fdet & \fdef & \fden\\ \hline
\fdeb & \fdeb & \fdeb & \fdef & \fdef\\
\fdet & \fdeb & \fdet & \fdef & \fden\\
\fdef & \fdef & \fdef & \fdef & \fdef\\
\fden & \fdef & \fden & \fdef & \fden
\end{array}\qquad
\begin{array}[t]{c|cccc}
\lor^\FDE & \fdeb & \fdet & \fdef & \fden\\ \hline
\fdeb & \fdeb & \fdet & \fdeb & \fdet\\
\fdet & \fdet & \fdet & \fdet & \fdet\\
\fdef & \fdeb & \fdet & \fdef & \fden\\
\fden & \fdet & \fdet & \fden & \fden
\end{array}\]

Fine's matrix is the direct product of \FDE{} with another
matrix~\ACT. \ACT{} has four truth values $\mathit{AC}_2 = \{\acb,
\act, \acf, \acn\}$ with designated values $\mathit{AC}_2^+ = \{\acf,
\acn\}$ and truth functions as follows:\footnote{Again, $\acb, \act,
\acf, \acn$ correspond to Fine's tf, t$\!\not\mathrm{f}$,
$\not\mathrm{t}$f, $\not\mathrm{t}${}$\!\not\mathrm{f}$.}
\[\begin{array}[t]{c|c}
\lnot^\ACT\\
\hline
\acb & \acb\\
\act & \acf\\
\acf & \act\\
\acn & \acn
\end{array}
\qquad\begin{array}[t]{c|cccc}
\land^\ACT & \acb & \act & \acf & \acn\\ \hline
\acb & \acb & \acb & \acb & \acb\\
\act & \acb & \act & \acb & \act\\
\acf & \acb & \acb & \acf & \acf\\
\acn & \acb & \act & \acf & \acn
\end{array}
\qquad
\begin{array}[t]{c|cccc}
\lor^\ACT & \acb & \act & \acf & \acn\\ \hline
\acb & \acb & \acb & \acb & \acb\\
\act & \acb & \act & \acb & \act\\
\acf & \acb & \acb & \acf & \acf\\
\acn & \acb & \act & \acf & \acn
\end{array}\]
(so the truth functions for $\land$ and $\lor$ are identical).

The resulting matrix $\FC = \FDE \times \ACT$ has sixteen truth values and
four designated values $\mathit{FC}^+ = \mathit{FDE}^+ \times
\mathit{AC}_2^+ = \{\fdet\acn,\fdeb\acf,\fdet\acf,\fdeb\acn\}$.

\section{Homomorphisms of matrices}

A function $f\colon L \to L'$ from the truth values of a
matrix~\textbf{L} to those of a matrix~$\textbf{L}'$ is a \emph{strong
homomorphism} if it respects truth functions and designated values,
that is,
\begin{enumerate}
    \item $f(\Box^\L(v_1, \dots, v_n)) =
    \Box^{\L'}(f(v_1), \dots, f(v_n))$, and
    \item $v \in L^+$ iff $f(v) \in L^+$.
\end{enumerate}
A surjective homomorphism is called an \emph{epimorphism}, and a bijective homomorphism an \emph{isomorphism}.

The following facts are elementary results of universal algebra or can
easily be verified (see, e.g., \citealt{Gratzer1968,Wojcicki1988}):
\begin{enumerate}
    \item A homomorphism $f$ induces a partition of~$L$ consisting of
    the sets $[v]_f = \{v' \mid f(v) = f(v')\}$.
    \item The equivalence relation~$\equiv_f$ on~$L$ corresponding to the
    partition induced by a homomorphism~$f$ is a congruence
    of~\textbf{L}, that is:
    \begin{enumerate}
    \item If $v_1 \equiv_f v_1'$, \dots, $v_n \equiv_f v_n'$ then
    $\Box^\L(v_1, \dots, v_n) \equiv_f \Box^\L(v_1',
    \dots, v_n')$, and
    \item if $v \equiv_f v'$ then $v \in L^+$ iff $v' \in L^+$.
    \end{enumerate}
    \item If $\equiv$ is a congruence on \L{}, then $f\colon L \to
    L/{\equiv}$ defined by $f(v) = [v]_f$ is an epimorphism of
    \L{} to the factor matrix~$\L/{\equiv}$.
    \item Every epimorphism $h\colon \L \to \L'$ factors into the
    epimorphism $f\colon \L \to \L/{\equiv_h}$ and an isomorphism
    $g\colon \L/{\equiv_h} \to \L'$.
    \item If $f\colon\L \to \L'$ is an epimorphism then the
    consequence relations of \L{} and~$\L'$ (defined via preservation
    of designated values) agree, i.e., $\Delta \vDash_{\L} \varphi$ iff
    $\Delta \vDash_{\L'} \varphi$.
\end{enumerate}

\section{The epimorphism of \FC{} to \NC}

In light of the above, to show that \NC{} is an epimorphic image
of~\FC, and hence that $\vDash_\NC$ coincides with $\vDash_\FC$, it
suffices to find a congruence relation~$\equiv$ of \FC{} and an
isomorphism between $\FC/{\equiv}$ and~\NC.

\FC{} has two non-trivial congruences, namely:
\begin{align*}
    & \fdeb\act, \fdef\acf, \fdef\acn, \fdef\act, \fden\acf,
    \fden\acn, \fden\act, \fdeb\acb \equiv \fdef\acb \equiv \fden\acb
    \equiv \fdet\acb, \fdet\act, \fdeb\acf, \fdeb\acn, \fdet\acf,
    \fdet\acn \\
    & \fden\acn, \fdef\acf \equiv \fden\acf, \fdef\acn, \fden\act
    \equiv
    \fdet\act, \fdeb\acb \equiv \fdef\acb \equiv \fden\acb \equiv
    \fdet\acb, \fdeb\act \equiv \fdef\act, \fdet\acn, \fdeb\acf \equiv
    \fdet\acf, \fdeb\acn
\end{align*}
The latter has nine equivalence classes, of which the three classes
$\{\fdet\acn\}$, $\{\fdeb\acf,\fdet\acf\}$, and $\{\fdeb\acn\}$ are
designated.  And in fact, the following mapping is an isomorphism
between \NC{} and $\FC/{\equiv}$:
\[\begin{array}{c|ccccccccc}
v &
\ncf\ncf & \ncf\ncu & \ncf\nct & \ncu\ncf & \ncu\ncu & \ncu\nct &
\nct\ncf & \nct\ncu & \nct\nct\\
g(v) &
\fden\acn & \fdef\acf, \fden\acf & \fdef\acn & \fden\act, \fdet\act &
\fdeb\acb, \fdef\acb,  & \fdeb\act, \fdef\act &
\fdet\acn & \fdeb\acf, \fdet\acf & \fdeb\acn\\
& & & & & \fden\acb, \fdet\acb\\
\hline
& & \fdef\fden\acf & & \fden\fdet\act & \fdeb{\cdot}\fden\acb &
\fdeb\fdef\act & & \fdeb\fdet\acf
\end{array}
\] The last line introduces abbreviations for the congruence
classes of~$\FC/{\equiv}$ consisting of more than one truth value. We
use these abbreviations merely to save space in tables \ref{table0}
and~\ref{table3} below.

The verification of the facts that $\equiv$ is a congruence
on~$\FC$ and that $g$ is an isomorphism between \NC{} and
$\FC/{\equiv}$ would be extremely tedious. We can make them
immediately apparent, however, by displaying the truth tables for each
logic in full, with truth values that are isomorphic or equivalent
shown in the same color, and by arranging the truth values in
corresponding order. For instance, the truth tables for $\lnot$ in
\NC, $\FC/{\equiv}$ and \FC, respectively, are given in
table~\ref{table0}.
\begin{table}
\def\mystrut{\rule[-.3ex]{0pt}{2.5ex}}
\[\begin{array}[t]{c|c}
\lnot^\NC & \\
\hline
\ncff & \ncff\\
\ncfu & \ncuf\\
\ncft & \nctf\\
\ncuf & \ncfu\\
\ncuu & \ncuu\\
\ncut & \nctu\\
\nctf & \ncft\\
\nctu & \ncut\\
\nctt & \nctt\\
\end{array}\qquad
\begin{array}[t]{c|c}
\lnot^{\FC/{\equiv}} & \\
\hline
\mystrut\acNn & \acNn\\
\mystrut\acFNf & \acNTt\\
\mystrut\acFn & \acTn\\
\mystrut\acNTt & \acFNf\\
\mystrut\acBFNTb & \acBFNTb\\
\mystrut\acBtFt & \acBTf\\
\mystrut\acTn & \acFn\\
\mystrut\acBTf & \acBtFt\\
\mystrut\acBn & \acBn\\
\end{array}
\qquad
\begin{array}[t]{c|c}
\lnot^{\FC} &\\
\hline
\fcnn & \fcnn\\
\fcff & \fctt\\
\fcnf & \fcnt\\
\fcfn & \fctn\\
\fcnt & \fcnf\\
\fctt & \fcff\\
\fcfb & \fctb\\
\fcnb & \fcnb\\
\fcbb & \fcbb\\
\fctb & \fcfb\\
\fcft & \fctf\\
\fcbt & \fcbf\\
\fctn & \fcfn\\
\fcbf & \fcbt\\
\fctf & \fcft\\
\fcbn & \fcbn\\
\end{array}
\] 
\caption{$\lnot$ in \NC, $\FC/{\equiv}$, and \FC}
\label{table0}
\end{table}
Reversed colors indicate the
designated values. Compare the tables for $\land$ in \NC{} and \FC{}
in table~\ref{table1}, and those for $\lor$ in \NC{} and \FC{} in
table~\ref{table2}.

\begin{table}
\[\begin{array}{c|c@{\,}c@{\,}c@{\,}c@{\,}c@{\,}c@{\,}c@{\,}c@{\,}c}\\
    \land^{\NC}  & \ncff & \ncfu & \ncft & \ncuf & \ncuu & \ncut & \nctf & \nctu & \nctt\\ \hline
\ncff & \ncff & \ncfu & \ncft & \ncuf & \ncuu & \ncut & \ncff & \ncfu & \ncft\\
\ncfu & \ncfu & \ncfu & \ncfu & \ncuu & \ncuu & \ncuu & \ncfu & \ncfu & \ncfu\\
\ncft & \ncft & \ncfu & \ncft & \ncut & \ncuu & \ncut & \ncft & \ncfu & \ncft\\
\ncuf & \ncuf & \ncuu & \ncut & \ncuf & \ncuu & \ncut & \ncuf & \ncuu & \ncut\\
\ncuu & \ncuu & \ncuu & \ncuu & \ncuu & \ncuu & \ncuu & \ncuu & \ncuu & \ncuu\\
\ncut & \ncut & \ncuu & \ncut & \ncut & \ncuu & \ncut & \ncut & \ncuu & \ncut\\
\nctf & \ncff & \ncfu & \ncft & \ncuf & \ncuu & \ncut & \nctf & \nctu & \nctt\\
\nctu & \ncfu & \ncfu & \ncfu & \ncuu & \ncuu & \ncuu & \nctu & \nctu & \nctu\\
\nctt & \ncft & \ncfu & \ncft & \ncut & \ncuu & \ncut & \nctt & \nctu & \nctt\\
\end{array}\]
\[\begin{array}{c|c@{\,}c@{\,}c@{\,}c@{\,}c@{\,}c@{\,}c@{\,}c@{\,}c@{\,}c@{\,}c@{\,}c@{\,}c@{\,}c@{\,}c@{\,}c}
    \cline{3-4}\cline{6-7}\cline{8-11}\cline{12-13}\cline{15-16}
    \land^{\FC} 
    & \fcnn & \lr\fcff & \rr\fcnf & \fcfn & \lr\fcnt & \rr\fctt & \lr\fcfb & \fcnb &
    \fcbb & \rr\fctb & \fcft & \rr\fcbt & \fctn & \lr\fcbf & \rr\fctf & \fcbn\\ 
    \cline{3-4}\cline{6-7}\cline{8-11}\cline{12-13}\cline{15-16}
    \\[-2ex]\hline\\[-2ex]
    \cline{3-4}\cline{6-7}\cline{8-11}\cline{12-13}\cline{15-16}
\fcnn & \fcnn & \lr\fcff & \rr\fcnf & \fcfn & \lr\fcnt & \rr\fcnt &
\lr\fcfb & \fcnb & \fcfb & \rr\fcnb & \lr\fcft & \rr\fcft & \fcnn &
\lr\fcff & \rr\fcnf & \fcfn\\
\hline
\lrr\fcff & \lrr\fcff & \lr\fcff & \rr\fcff & \lrr\fcff & \lr\fcfb & \rr\fcfb &
\lr\fcfb & \fcfb & \fcfb & \rr\fcfb & \lr\fcfb & \rr\fcfb & \lrr\fcff & \lr\fcff & \rr\fcff
& \lrr\fcff\\
\lrr\fcnf & \lrr\fcnf & \lr\fcff & \rr\fcnf & \lrr\fcff & \lr\fcnb & \rr\fcnb &
\lr\fcfb & \fcnb & \fcfb & \rr\fcnb & \lr\fcfb & \rr\fcfb & \lrr\fcnf & \lr\fcff & \rr\fcnf
& \lrr\fcff\\
\hline
\fcfn & \fcfn & \lr\fcff & \rr\fcff & \fcfn & \lr\fcft & \rr\fcft &
\lr\fcfb & \fcfb & \fcfb & \rr\fcfb & \lr\fcft & \rr\fcft & \fcfn &
\lr\fcff & \rr\fcff & \fcfn\\
\hline
\lrr\fcnt & \lrr\fcnt & \lr\fcfb & \rr\fcnb & \lrr\fcft & \lr\fcnt & \rr\fcnt & \lr\fcfb & \fcnb & \fcfb & \rr\fcnb & \lr\fcft & \rr\fcft & \lrr\fcnt & \lr\fcfb & \rr\fcnb & \lrr\fcft\\
\lrr\fctt & \lrr\fcnt & \lr\fcfb & \rr\fcnb & \lrr\fcft & \lr\fcnt & \rr\fctt & \lr\fcfb & \fcnb & \fcbb & \rr\fctb & \lr\fcft & \rr\fcbt & \lrr\fctt & \lr\fcbb & \rr\fctb & \lrr\fcbt\\\hline
\lrr\fcfb & \lrr\fcfb & \lr\fcfb & \rr\fcfb & \lrr\fcfb & \lr\fcfb & \rr\fcfb & \lr\fcfb & \fcfb & \fcfb & \rr\fcfb & \lr\fcfb & \rr\fcfb & \lrr\fcfb & \lr\fcfb & \rr\fcfb & \lrr\fcfb\\
\lrr\fcnb & \lrr\fcnb & \lr\fcfb & \rr\fcnb & \lrr\fcfb & \lr\fcnb & \rr\fcnb & \lr\fcfb & \fcnb & \fcfb & \rr\fcnb & \lr\fcfb & \rr\fcfb & \lrr\fcnb & \lr\fcfb & \rr\fcnb & \lrr\fcfb\\
\lrr\fcbb & \lrr\fcfb & \lr\fcfb & \rr\fcfb & \lrr\fcfb & \lr\fcfb & \rr\fcbb & \lr\fcfb & \fcfb & \fcbb & \rr\fcbb & \lr\fcfb & \rr\fcbb & \lrr\fcbb & \lr\fcbb & \rr\fcbb & \lrr\fcbb\\
\lrr\fctb & \lrr\fcnb & \lr\fcfb & \rr\fcnb & \lrr\fcfb & \lr\fcnb & \rr\fctb & \lr\fcfb & \fcnb & \fcbb & \rr\fctb & \lr\fcfb & \rr\fcbb & \lrr\fctb & \lr\fcbb & \rr\fctb & \lrr\fcbb\\\hline
\lrr\fcft & \lrr\fcft & \lr\fcfb & \rr\fcfb & \lrr\fcft & \lr\fcft & \rr\fcft & \lr\fcfb & \fcfb & \fcfb & \rr\fcfb & \lr\fcft & \rr\fcft & \lrr\fcft & \lr\fcfb & \rr\fcfb & \lrr\fcft\\
\lrr\fcbt & \lrr\fcft & \lr\fcfb & \rr\fcfb & \lrr\fcft & \lr\fcft & \rr\fcbt & \lr\fcfb & \fcfb & \fcbb & \rr\fcbb & \lr\fcft & \rr\fcbt & \lrr\fcbt & \lr\fcbb & \rr\fcbb & \lrr\fcbt\\\hline
\fctn & \fcnn & \lr\fcff & \rr\fcnf & \fcfn & \lr\fcnt & \rr\fctt & \lr\fcfb & \fcnb & \fcbb & \rr\fctb & \lr\fcft & \rr\fcbt & \fctn & \lr\fcbf & \rr\fctf & \fcbn\\\hline
\lrr\fcbf & \lrr\fcff & \lr\fcff & \rr\fcff & \lrr\fcff & \lr\fcfb & \rr\fcbb & \lr\fcfb & \fcfb & \fcbb & \rr\fcbb & \lr\fcfb & \rr\fcbb & \lrr\fcbf & \lr\fcbf & \rr\fcbf & \lrr\fcbf\\
\lrr\fctf & \lrr\fcnf & \lr\fcff & \rr\fcnf & \lrr\fcff & \lr\fcnb & \rr\fctb & \lr\fcfb & \fcnb & \fcbb & \rr\fctb & \lr\fcfb & \rr\fcbb & \lrr\fctf & \lr\fcbf & \rr\fctf & \lrr\fcbf\\\hline
\fcbn & \fcfn & \lr\fcff & \rr\fcff & \fcfn & \lr\fcft & \rr\fcbt &
\lr\fcfb & \fcfb & \fcbb & \rr\fcbb & \lr\fcft & \rr\fcbt & \fcbn &
\lr\fcbf & \rr\fcbf & \fcbn\\
\cline{3-4}\cline{6-7}\cline{8-11}\cline{12-13}\cline{15-16}
\end{array}\]
\caption{$\land$ in \NC{} and \FC. Values in any rectangle are
equivalent and correspond to a single value in \NC.}
\label{table1}
\end{table}

\begin{table}
\[\begin{array}{c|c@{\,}c@{\,}c@{\,}c@{\,}c@{\,}c@{\,}c@{\,}c@{\,}c}\\
\lor^{\NC} & \ncff & \ncfu & \ncft & \ncuf & \ncuu & \ncut & \nctf & \nctu & \nctt\\ \hline
\ncff & \ncff & \ncfu & \ncff & \ncuf & \ncuu & \ncuf & \nctf & \nctu & \nctf\\
\ncfu & \ncfu & \ncfu & \ncfu & \ncuu & \ncuu & \ncuu & \nctu & \nctu & \nctu\\
\ncft & \ncff & \ncfu & \ncft & \ncuf & \ncuu & \ncut & \nctf & \nctu & \nctt\\
\ncuf & \ncuf & \ncuu & \ncuf & \ncuf & \ncuu & \ncuf & \ncuf & \ncuu & \ncuf\\
\ncuu & \ncuu & \ncuu & \ncuu & \ncuu & \ncuu & \ncuu & \ncuu & \ncuu & \ncuu\\
\ncut & \ncuf & \ncuu & \ncut & \ncuf & \ncuu & \ncut & \ncuf & \ncuu & \ncut\\
\nctf & \nctf & \nctu & \nctf & \ncuf & \ncuu & \ncuf & \nctf & \nctu & \nctf\\
\nctu & \nctu & \nctu & \nctu & \ncuu & \ncuu & \ncuu & \nctu & \nctu & \nctu\\
\nctt & \nctf & \nctu & \nctt & \ncuf & \ncuu & \ncut & \nctf & \nctu & \nctt\\
\end{array}\]
\[\begin{array}{c|c@{\,}c@{\,}c@{\,}c@{\,}c@{\,}c@{\,}c@{\,}c@{\,}c@{\,}c@{\,}c@{\,}c@{\,}c@{\,}c@{\,}c@{\,}c}\\
    \cline{3-4}\cline{6-7}\cline{8-11}\cline{12-13}\cline{15-16}
\lor^{\FC} & \fcnn & \lr\fcff & \rr\fcnf & \fcfn & \lr\fcnt & \rr\fctt & \lr\fcfb &
\fcnb & \fcbb & \rr\fctb & \lr\fcft & \rr\fcbt & \fctn & \lr\fcbf & \rr\fctf &
\fcbn\\ 
\cline{3-4}\cline{6-7}\cline{8-11}\cline{12-13}\cline{15-16}
\\[-2ex]\hline\\[-2ex]
\cline{3-4}\cline{6-7}\cline{8-11}\cline{12-13}\cline{15-16}
\fcnn & \fcnn & \lr\fcnf & \rr\fcnf & \fcnn & \lr\fcnt & \rr\fctt &
\lr\fcnb & \fcnb & \fctb & \rr\fctb & \lr\fcnt & \rr\fctt & \fctn &
\lr\fctf & \rr\fctf & \fctn\\ \hline
\lrr\fcff & \lrr\fcnf & \lr\fcff & \rr\fcnf & \lrr\fcff & \lr\fcnb & \rr\fctb & \lr\fcfb & \fcnb & \fcbb & \rr\fctb & \lr\fcfb & \rr\fcbb & \lrr\fctf & \lr\fcbf & \rr\fctf & \lrr\fcbf\\
\lrr\fcnf & \lrr\fcnf & \lr\fcnf & \rr\fcnf & \lrr\fcnf & \lr\fcnb &
\rr\fctb & \lr\fcnb & \fcnb & \fctb & \rr\fctb & \lr\fcnb & \rr\fctb &
\lrr\fctf & \lr\fctf & \rr\fctf & \lrr\fctf\\ \hline
\fcfn & \fcnn & \lr\fcff & \rr\fcnf & \fcfn & \lr\fcnt & \rr\fctt &
\lr\fcfb & \fcnb & \fcbb & \rr\fctb & \lr\fcft & \rr\fcbt & \fctn &
\lr\fcbf & \rr\fctf & \fcbn\\ \hline
\lrr\fcnt & \lrr\fcnt & \lr\fcnb & \rr\fcnb & \lrr\fcnt & \lr\fcnt & \rr\fctt & \lr\fcnb & \fcnb & \fctb & \rr\fctb & \lr\fcnt & \rr\fctt & \lrr\fctt & \lr\fctb & \rr\fctb & \lrr\fctt\\
\lrr\fctt & \lrr\fctt & \lr\fctb & \rr\fctb & \lrr\fctt & \lr\fctt &
\rr\fctt & \lr\fctb & \fctb & \fctb & \rr\fctb & \lr\fctt & \rr\fctt &
\lrr\fctt & \lr\fctb & \rr\fctb & \lrr\fctt\\ \hline
\lrr\fcfb & \lrr\fcnb & \lr\fcfb & \rr\fcnb & \lrr\fcfb & \lr\fcnb & \rr\fctb & \lr\fcfb & \fcnb & \fcbb & \rr\fctb & \lr\fcfb & \rr\fcbb & \lrr\fctb & \lr\fcbb & \rr\fctb & \lrr\fcbb\\
\lrr\fcnb & \lrr\fcnb & \lr\fcnb & \rr\fcnb & \lrr\fcnb & \lr\fcnb & \rr\fctb & \lr\fcnb & \fcnb & \fctb & \rr\fctb & \lr\fcnb & \rr\fctb & \lrr\fctb & \lr\fctb & \rr\fctb & \lrr\fctb\\
\lrr\fcbb & \lrr\fctb & \lr\fcbb & \rr\fctb & \lrr\fcbb & \lr\fctb & \rr\fctb & \lr\fcbb & \fctb & \fcbb & \rr\fctb & \lr\fcbb & \rr\fcbb & \lrr\fctb & \lr\fcbb & \rr\fctb & \lrr\fcbb\\
\lrr\fctb & \lrr\fctb & \lr\fctb & \rr\fctb & \lrr\fctb & \lr\fctb & \rr\fctb & \lr\fctb & \fctb & \fctb & \rr\fctb & \lr\fctb & \rr\fctb & \lrr\fctb & \lr\fctb & \rr\fctb & \lrr\fctb\\\hline
\lrr\fcft & \lrr\fcnt & \lr\fcfb & \rr\fcnb & \lrr\fcft & \lr\fcnt & \rr\fctt & \lr\fcfb & \fcnb & \fcbb & \rr\fctb & \lr\fcft & \rr\fcbt & \lrr\fctt & \lr\fcbb & \rr\fctb & \lrr\fcbt\\
\lrr\fcbt & \lrr\fctt & \lr\fcbb & \rr\fctb & \lrr\fcbt & \lr\fctt & \rr\fctt & \lr\fcbb & \fctb & \fcbb & \rr\fctb & \lr\fcbt & \rr\fcbt & \lrr\fctt & \lr\fcbb & \rr\fctb & \lrr\fcbt\\\hline
\fctn & \fctn & \lr\fctf & \rr\fctf & \fctn & \lr\fctt & \rr\fctt & \lr\fctb & \fctb & \fctb & \rr\fctb & \lr\fctt & \rr\fctt & \fctn & \lr\fctf & \rr\fctf & \fctn\\\hline
\lrr\fcbf & \lrr\fctf & \lr\fcbf & \rr\fctf & \lrr\fcbf & \lr\fctb & \rr\fctb & \lr\fcbb & \fctb & \fcbb & \rr\fctb & \lr\fcbb & \rr\fcbb & \lrr\fctf & \lr\fcbf & \rr\fctf & \lrr\fcbf\\
\lrr\fctf & \lrr\fctf & \lr\fctf & \rr\fctf & \lrr\fctf & \lr\fctb & \rr\fctb & \lr\fctb & \fctb & \fctb & \rr\fctb & \lr\fctb & \rr\fctb & \lrr\fctf & \lr\fctf & \rr\fctf & \lrr\fctf\\\hline
\fcbn & \fctn & \lr\fcbf & \rr\fctf & \fcbn & \lr\fctt & \rr\fctt & \lr\fcbb & \fctb & \fcbb & \rr\fctb & \lr\fcbt & \rr\fcbt & \fctn & \lr\fcbf & \rr\fctf & \fcbn\\
\cline{3-4}\cline{6-7}\cline{8-11}\cline{12-13}\cline{15-16}
\end{array}\]
\caption{$\lor$ in \NC{} and \FC}
\label{table2}
\end{table}

\begin{table}
\[\begin{array}{c|c@{\,}c@{\,}c@{\,}c@{\,}c@{\,}c@{\,}c@{\,}c@{\,}c}\\
\land^{\FC/{\equiv}} & \acNn & \acFNf & \acFn & \acNTt & \acBFNTb & \acBtFt & \acTn & \acBTf & \acBn\\ \hline
\acNn & \acNn & \acFNf & \acFn & \acNTt & \acBFNTb & \acBtFt & \acNn & \acFNf & \acFn\\
\acFNf & \acFNf & \acFNf & \acFNf & \acBFNTb & \acBFNTb & \acBFNTb & \acFNf & \acFNf & \acFNf\\
\acFn & \acFn & \acFNf & \acFn & \acBtFt & \acBFNTb & \acBtFt & \acFn & \acFNf & \acFn\\
\acNTt & \acNTt & \acBFNTb & \acBtFt & \acNTt & \acBFNTb & \acBtFt & \acNTt & \acBFNTb & \acBtFt\\
\acBFNTb & \acBFNTb & \acBFNTb & \acBFNTb & \acBFNTb & \acBFNTb & \acBFNTb & \acBFNTb & \acBFNTb & \acBFNTb\\
\acBtFt & \acBtFt & \acBFNTb & \acBtFt & \acBtFt & \acBFNTb & \acBtFt & \acBtFt & \acBFNTb & \acBtFt\\
\acTn & \acNn & \acFNf & \acFn & \acNTt & \acBFNTb & \acBtFt & \acTn & \acBTf & \acBn\\
\acBTf & \acFNf & \acFNf & \acFNf & \acBFNTb & \acBFNTb & \acBFNTb & \acBTf & \acBTf & \acBTf\\
\acBn & \acFn & \acFNf & \acFn & \acBtFt & \acBFNTb & \acBtFt & \acBn & \acBTf & \acBn\\
\end{array}\]
\[\begin{array}{c|c@{\,}c@{\,}c@{\,}c@{\,}c@{\,}c@{\,}c@{\,}c@{\,}c}\\
\lor^{\FC/{\equiv}} & \acNn & \acFNf & \acFn & \acNTt & \acBFNTb & \acBtFt & \acTn & \acBTf & \acBn\\ \hline
\acNn & \acNn & \acFNf & \acNn & \acNTt & \acBFNTb & \acNTt & \acTn & \acBTf & \acTn\\
\acFNf & \acFNf & \acFNf & \acFNf & \acBFNTb & \acBFNTb & \acBFNTb & \acBTf & \acBTf & \acBTf\\
\acFn & \acNn & \acFNf & \acFn & \acNTt & \acBFNTb & \acBtFt & \acTn & \acBTf & \acBn\\
\acNTt & \acNTt & \acBFNTb & \acNTt & \acNTt & \acBFNTb & \acNTt & \acNTt & \acBFNTb & \acNTt\\
\acBFNTb & \acBFNTb & \acBFNTb & \acBFNTb & \acBFNTb & \acBFNTb & \acBFNTb & \acBFNTb & \acBFNTb & \acBFNTb\\
\acBtFt & \acNTt & \acBFNTb & \acBtFt & \acNTt & \acBFNTb & \acBtFt & \acNTt & \acBFNTb & \acBtFt\\
\acTn & \acTn & \acBTf & \acTn & \acNTt & \acBFNTb & \acNTt & \acTn & \acBTf & \acTn\\
\acBTf & \acBTf & \acBTf & \acBTf & \acBFNTb & \acBFNTb & \acBFNTb & \acBTf & \acBTf & \acBTf\\
\acBn & \acTn & \acBTf & \acBn & \acNTt & \acBFNTb & \acBtFt & \acTn & \acBTf & \acBn\\
\end{array}\]
\caption{$\land$ and $\lor$ in $\FC/{\equiv}$}
\label{table3}
\end{table}

\section{Complexity}

The questions of whether a matrix is an epimorphic image of another,
and whether two matrices are isomorphic, are computationally
non-trivial to answer.  The number of different bijective mappings
between two $n$-valued matrices equals the number of possible
permutations of the $n$~truth values, i.e.,~$n!$.  However, since a
potential isomorphism has to respect the designated truth values, the
number of maps to be checked is actually ``only'' $k!(n-k)!$ where $k$
is the number of designated values. In our case, this amounts to
$3!\times 6! = 4{,}320$ potential isomorphisms; already too large to
be checked manually by brute force.  To find a candidate epimorphic
image in the first place, an exhaustive search would have to enumerate
all potential partitions and check if they are congruences. The number
of partitions of a size~$k$ set is $B_k$, the $k$-th Bell number.
Since congruences have to again respect the designated values, the
number of potential congruences to be checked can be reduced to
$B_kB_{n-k}$ (for an $n$-valued matrix with $k$ designated values). In
the case of~\FC, there are $15 \times 4{,}213{,}597 = 63{,}203{,}955$
potential congruences. Thus, finding all epimorphic images of~\FC,
and verifying that \NC{} is one, is impossible to do by brute force
without the help of a computer.

Note that the approach taken is already much better than the
completely naive approach of checking every surjective function from
$\mathit{FDE} \times \mathit{AC}_2$ to $\mathit{NC}$ for whether it is
a homomorphism. First of all, this would require the verification of
many more candidate mappings. The number of different surjective
functions from an $n$-element set to one of size~$m$ is $m!S(n,m)$
(where $S(n,m)$ are the Stirling numbers of the second kind). Since
any epimorphism must respect designated values, we can restrict the
possible values for designated and non-designated arguments. The total
number of candidates in our case would be $6! S(12,6) \times 3!
S(4,3)$, or over 34 billion. Our approach also provides more
information: since \FC{} has no congruence with fewer than
$9$~classes, there can be no smaller common factor of both \NC{}
and~\FC. This cannot be ruled out a priori, and would have yielded an
interesting result---a matrix with the same consequence relation as
\NC{} and~\FC, but fewer truth values than either. Merely checking if
an epimorphism from \FC{} to \NC{} exists would not have settled that
question.

It is in practice not necessary to run through the entirety of all
partitions or all bijections to find congruences and isomorphisms,
respectively.  A simple idea will cut down the search space to
manageable size.  Here's the idea for isomorphisms: call an injection
$f\colon U \to V'$ (where $U \subseteq V$, and $V$, $V'$ are the truth
value sets of \L{} and~$\L'$, respectively) a \emph{partial
isomorphism} if it respects operations, i.e., if $\Box^{\L'}(f(v_1),
\dots, f(v_n)) = f(\Box^\L(v_1, \dots, v_n))$ provided $v_i$ and
$\Box(v_1, \dots, v_n) \in \dom(f)$.\footnote{It must also respect
designated values, but we can simply split the bijection into a
bijection between designated values and one between undesignated
values as before.} A bijective $f\colon V \to V'$ is an isomorphism
iff $f \upharpoonright U$ is a partial isomorphism for all $U \subseteq V$. So to
find an isomorphism we can proceed as follows: Set $f_0 = \emptyset$
and let $f_{i+1} = f_i \cup (v,v')$ where $v \in V\setminus \dom(f_i)$
and $v' \in V'\setminus \ran(f_i)$. If $f_i$ is not a partial
isomorphism, no expansion of~$f_i$ can be an isomorphism; in that
case, backtrack and pick a different pair $(v, v')$. In other words,
instead of generating and testing all isomorphisms, generate
isomorphisms one value at a time. If a particular choice of value
results in a conflict with the truth tables, it is guaranteed that no
expansion of that sequence of choices to a total bijection is an
isomorphism.  A similar idea can be used to speed up the search for
congruences.  Nevertheless, even if there are just a handful of
partial isomorphisms to test, this is intractable by hand. For
instance, verifying that \NC{} with its modest $9$~truth values and
three operations has no nontrivial automorphisms requires checking
$27$ partial isomorphisms and computing almost $1{,}400$ individual
operations.

\section{Application of MUltlog}

Finding the congruences of~\FC, i.e., its potential factor matrices,
and the verification that $\FC/{\equiv}$ is isomorphic to~\NC, was
accomplished using the MUltlog system due to
\citet{Salzer1996a}.\footnote{The MUltlog software is available at
\href{https://logic.at/multlog}{logic.at/multlog}. Results reported in
this paper were generated on version 1.16a \citep{Salzer2022}. The
supplementary code used is available at
\href{https://github.com/rzach/ncac}{github.com/rzach/ncac}
and is archived with \DOI{https://doi.org/}{10.5281/zenodo.6893328}.}  MUltlog was
originally designed to compute optimized $n$-sided sequent calculus
rules for arbitrary $n$-valued matrices using the methods of
\cite{BaazFermullerZach1994} and \cite{Salzer2000}.  As of version
v1.5, MUltlog added interactive functionality, including: evaluating
formulas in a matrix, testing for and finding tautologies of matrices,
defining products and factors of matrices, finding congruences of
matrices, and checking for isomorphism between
matrices.\footnote{Implementation of these features was carried out,
for the most part, by the author.} The solution of Fine's question
presented above was found by:
\begin{enumerate}
    \item Specifying the matrices for \NC, \FDE, and \ACT{} in
    MUltlog's format;\footnote{See appendix~\ref{lgcs}.}
    \item Defining the product \FC{};
    \item Finding the epimorphism from \FC{} to~\NC{}.
\end{enumerate}
Only the specification of \NC{} required substantive intervention.
MUltlog provided the basis for the \LaTeX{} code of the truth tables
with matching orders of truth values and matching
colors.\footnote{Finding the epimorphism between \FC{} and~\NC{}
required less than a second. Checking all potential congruences
of~\FC{} took about 20 minutes. So had we not had the candidate
matrix~\NC{} available, we could have found it automatically. The
computations were carried out on an Asus Zenbook 14 (Intel Core
i7--8565U CPU, 16GB memory), running Ubuntu Linux 21.10 and SWI-Prolog
8.2.4.} The algorithms used are still relatively naive.  Much more
efficient algorithms exist to solve questions like this, e.g., see
\cite{Freese2008} for a fast algorithm to find a minimal congruence of
a finite algebra.  The naive methods, however, have the advantage of
being straightforwardly implemented in Prolog, and an extension of
MUltlog provides a convenient way to operate on finite-valued
matrices.

MUltlog's original purpose, as mentioned, is to compute inference
rules for $n$-valued logics.  It does this not just for propositional
operators but also for so-called distribution quantifiers. Any
associative, commutative, idempotent binary operation induces such a
quantifier. E.g., if in a given interpretation, the formula $A(x)$
takes all and only values in $\{v_1, \dots, v_n\}$ then the induced
quantifier $\forall x\,A(x)$ of, say, $\land$, takes the value $v_1
\land v_2 \land \dots \land v_n$. Since \NC's $\land$ and $\lor$ are
associative, commutative, and idemptotent, \NC{} has universal and
existential quantifiers that generalize these connectives. Their truth
tables are unwieldy, as they list the value of $\forall x\,A(x)$ for
all potential distributions of $A(x)$, i.e., all $2^9-1 = 511$
possible non-empty subsets of~$\mathit{NC}$.\footnote{All but 113 of
which result in the value $\ncu\ncu$.} MUltlog nevertheless finds
optimal inference rules in minutes.  The tableaux calculus for \NC{}
generated by MUltlog can be found in
appendix~\ref{sec:tableaux}.\footnote{\citet{Ferguson2021} has
recently given a much more elegant tableau system for~\NC. }

\section{The mystery of the 7-valued matrix}

\citet[p.~223]{Fine2016} suggested that ``The 16 values [of \FC] can,
in fact, be reduced to 7 since we may just differentiate one
designated value, $\fdeb$ or $\fdet$, into four values, when paired
with the values $\acb$, $\act$, $\acf$, or $\acn$.''  It is not
perfectly clear what he has in mind here, but one way of making it
precise is this: If we let $V$ be either $\fdeb$ or $\fdet$, then we
can differentiate $V$ into the four values $V\acb$, $V\act$, $V\acf$,
and~$V\acn$. The remaining three truth values in $\{\fdeb, \fdet,
\fdef, \fden\} \setminus \{V\}$ would then be paired with a fixed
value $v \in \{\acb, \act, \acf, \acn\}$, resulting in 7~values
altogether. In other words, the truth values~$\mathit{FC}_{Vv}$ of
$\FC_{Vv}$ are
\[
\mathit{FC}_{Vv} = (\{V\} \times \mathit{AC}_2) \cup ((\mathit{FDE}
\setminus \{V\}) \times \{v\}).
\]
The truth functions of the corresponding matrix $\FC_{Vv}$ would be
given by
\begin{align*}
   \lnot^{\FC_{Vv}}(\langle v_0, v_1\rangle) &= h_{Vv}(\langle\lnot^\FDE(v_0), \lnot^\ACT(v_1)\rangle)\\
   \land^{\FC_{Vv}}(\langle v_0, v_1\rangle, \langle v_0', v_1'\rangle) & = h_{Vv}(\langle\land^\FDE(v_0, v_0'), \land^\ACT(v_1, v_1')\rangle)\\
   \lor^{\FC_{Vv}}(\langle v_0, v_1\rangle, \langle v_0', v_1'\rangle) & = h_{Vv}(\langle\lor^\FDE(v_0, v_0'), \lor^\ACT(v_1, v_1')\rangle)
\intertext{where the function $h_{Vv}$ guarantees that the result is in the corresponding set of truth values, by letting}
h_{Vv}(\langle v_0, v_1\rangle) &= \begin{cases}
   \langle v_0, v_1\rangle & \text{if $v_0 = V$}\\
   \langle v_0, v\rangle & \text{otherwise.}
\end{cases}
\end{align*}
In other words, $\FC_{Vv}$ is like $\FC$, except that only $V$ is
paired with all four truth values of~$\ACT$, and the remaining truth
values of $\FDE$ are each paired with the single value~$v$ of~$\ACT$.
The value $v$ is ``dominant,'' in the sense that if the value of one
of the truth functions of $\FC$ would yield a value $\langle v_0,
v_1\rangle$ where $v_0 \neq V$ and $v_1 \neq v$, then the value taken
is $\langle v_0, v\rangle$ instead of $\langle v_0, v_1\rangle$.

This leaves the question of which truth values to designate
in~$\FC_{Vv}$. It is most natural to let $\mathit{FC}_{Vv}^+ =
\mathit{FC}^+ \cap \mathit{FC}_{Vv}$, i.e., a truth value is
designated in~$\FC_{Vv}$ iff it is designated in~$\FC$. This gives us
eight different matrices, one for each of the possible ways of
combining either $\fdeb$ or~$\fdet$ with one of the four values
of~$\ACT$.

Let $V' = \fdeb$ if $V = \fdet$, and $= \fdet$ if $V = \fdeb$ (so, the
designated value of~$\FDE$ that's not $V$).  If $v \in
\mathit{AC}_2^+$, i.e., if $v = \acf$ or $= \acn$, then $V'v$ is
designated in $\FC_{Vv}$, and otherwise it is not. But $V'v$ has to
play the  role of \emph{all} the pairs $V'u$ with $u \in
\mathit{AC}_2$. So we should consider designating $V'v$ even when $v
\notin \mathit{AC}_2^+$, and also consider \emph{not} designating it
when $v \in \mathit{AC}^+$. So, let the matrix $\FC_{Vv}^*$ be just
like $\FC_{Vv}$ except that $V'v$ is designated if $v = \act$ or $=
\acb$, or undesignated if $v = \acf$ or $= \acn$.  For instance, in
$\FC_{\fdet\acb}$, the truth value $\fdeb\acb$ is not designated
(since $\acb \notin \mathit{AC}^+$). But, e.g, the designated value
$\fdeb\acf \in \mathit{FC}^+$ doesn't exist in $\FC_{\fdet\acb}$. If
$\fdeb\acb$ has to play its role in $\FC_{\fdet\acb}$, it would have
to be designated after all. So in $\FC_{\fdet\acb}^*$, the designated
values are $\fdet\acf$, $\fdet\acn$, and $\fdeb\acb$.  Conversely, if
$v \in \mathit{AC}_2^+$, the value $V'v$ is designated in~$\FC_{Vv}$.
We can choose to not designate it, since it has to play the role of
all values~$V'u$. E.g., the designated values of $\FC^*_{\fdet\acf}$
are $\fdet\acf$, $\fdet\acn$, but not $\fdeb\acf$. In general, if $v
\in \mathit{AC}^+$, then $\FC_{Vv}$ has three designated values while
$\FC_{Vv}^*$ has two; otherwise $\FC_{Vv}$ has two designated values
and $\FC_{Vv}^*$ has three. This results in another eight matrices.

For all of the resulting logics other than the four corresponding to
$\fdeb\act$ and~$\fdeb\acb$, already one of the axioms of~\AC{} given
in \citet{Ferguson2016} fails. For instance, in $\FC_{\fdet\acf}^*$,
$\fdet\acf$ is designated but $\lnot\lnot \fdet\acf = \fdet\act$ is
not, so $A \nvDash \lnot\lnot A$ and thus axiom AC1a fails. 

The 7-valued logics $\FC_{\fdeb\acb}$, $\FC_{\fdeb\acb}^*$,
$\FC_{\fdeb\act}$, and $\FC_{\fdeb\act}^*$, however, are sound for
those axioms. To show that here the reverse inclusions do not hold, we
have to find examples of entailments $\varphi \models \psi$ that hold in
the logics considered but not in~$\AC$, and this can be done by
finding countervaluations in~$\NC$. For $V = \fdeb$ and $v = \act$ or
$v = \acb$ we have $A \lor B \vDash_{\FC_{Vv}} B$ and $A \land (B \lor
\lnot A) \vDash_{\FC_{Vv}^*} B$. Neither consequence holds in~\NC. For
the first, observe that $\nct\ncf \lor \ncf\ncf = \nct\ncf$ is
designated in~$\NC$, but $\ncf\ncf$ is not. For the second, we have
$\nct\nct \land (\ncf\ncf \lor \lnot \nct\nct) = \nct\nct$ is
designated in~\NC, but $\ncf\ncf$, again, isn't.\footnote{These
counterexamples were found by MUltlog as well. $\AC_2$ happens to be
an epimorphic image of~$\FC_{\fdeb\acb}$. No two of the resulting
logics are isomorphic, and other than $\FC_{\fdeb\acb}$, none of them
have epimorphisms to \FDE{} or~$\AC_2$.} In short, none of the sixteen
potential 7-valued matrices in fact characterize~\AC.

Whether a matrix with fewer than \NC's nine truth values
characterizes~\AC{} remains an open question.  Fine's suggested
7-valued matrices do not. If a smaller matrix exists, it will have to
be proved to be adequate by independent methods, e.g., by a direct
soundness and completeness proof (they cannot be epimorphic images of
\NC{} or~\FC, as these have no congruences with fewer than nine
classes).

\bigskip
\noindent{\bseries Acknowledgments.} I would like to thank an
anonymous referee for their helpful suggestions.


\bigskip

\renewcommand{\bibfont}{\small}
\setlength{\bibsep}{6pt}

\AuthorAdressEmail{Richard Zach}{Department of Philosophy\\ University
of Calgary\\
Calgary, Canada}{rzach@ucalgary.ca\\richardzach.org}

\appendix
\section{MUltlog specification files}\label{lgcs}

\subsection*{Specification of \NC}
\begin{lstlisting}[basicstyle=\footnotesize\ttfamily,breaklines]
logic "NC".

truth_values { ff , fu , ft , uf , uu , ut , tf , tu , tt }.
designated_truth_values { tf , tu , tt }.

operator(neg/1, mapping {
	 (ff) : ff,
	 (fu) : uf,
	 (ft) : tf,
	 (uf) : fu,
	 (uu) : uu,
	 (ut) : tu,
	 (tf) : ft,
	 (tu) : ut,
	 (tt) : tt
	}
).

operator(and/2, table [
	       ff, fu, ft, uf, uu, ut, tf, tu, tt,
	  ff,  ff, fu, ft, uf, uu, ut, ff, fu, ft,
	  fu,  fu, fu, fu, uu, uu, uu, fu, fu, fu,
	  ft,  ft, fu, ft, ut, uu, ut, ft, fu, ft,
	  uf,  uf, uu, ut, uf, uu, ut, uf, uu, ut,
	  uu,  uu, uu, uu, uu, uu, uu, uu, uu, uu,
	  ut,  ut, uu, ut, ut, uu, ut, ut, uu, ut,
	  tf,  ff, fu, ft, uf, uu, ut, tf, tu, tt,
	  tu,  fu, fu, fu, uu, uu, uu, tu, tu, tu,
	  tt,  ft, fu, ft, ut, uu, ut, tt, tu, tt
	]
).

operator(or/2, table [
	       ff, fu, ft, uf, uu, ut, tf, tu, tt,
	  ff,  ff, fu, ff, uf, uu, uf, tf, tu, tf,
	  fu,  fu, fu, fu, uu, uu, uu, tu, tu, tu,
	  ft,  ff, fu, ft, uf, uu, ut, tf, tu, tt,
	  uf,  uf, uu, uf, uf, uu, uf, uf, uu, uf,
	  uu,  uu, uu, uu, uu, uu, uu, uu, uu, uu,
	  ut,  uf, uu, ut, uf, uu, ut, uf, uu, ut,
	  tf,  tf, tu, tf, uf, uu, uf, tf, tu, tf,
	  tu,  tu, tu, tu, uu, uu, uu, tu, tu, tu,
	  tt,  tf, tu, tt, uf, uu, ut, tf, tu, tt
	]
).

quantifier(forall, induced_by and/2).
quantifier(exists, induced_by or/2).
\end{lstlisting}
\subsection*{Specification of \FDE}
\begin{lstlisting}[basicstyle=\footnotesize\ttfamily,breaklines]

logic "FDE".

truth_values { b, t, f, n }.
designated_truth_values { b, t }.

ordering(truth, "f < {n, b} < t").

operator(neg/1, mapping {
	 (t) : f,
	 (b) : b,
	 (n) : n,
	 (f) : t
	}
).

operator(and/2, inf(truth)).
operator(or/2, sup(truth)).
\end{lstlisting}
\subsection*{Specification of \ACT}
\begin{lstlisting}[basicstyle=\footnotesize\ttfamily,breaklines]
logic "AC2".

truth_values { b, t, f, n }.
designated_truth_values { n , f }.

operator(neg/1, mapping {
	 (t) : f,
	 (b) : b,
	 (n) : n,
	 (f) : t
	}
).

operator(or/2, table [
	      t, b, n, f,
	  t,  t, b, t, b,
	  b,  b, b, b, b,
	  n,  t, b, n, f,
	  f,  b, b, f, f
	]
).

operator(and/2, table [
	      t, b, n, f,
	  t,  t, b, t, b,
	  b,  b, b, b, b,
	  n,  t, b, n, f,
	  f,  b, b, f, f
	]
).
\end{lstlisting}

\section{Tableaux for first-order \NC}\label{sec:tableaux}

\begin{definition}
   A \emph{signed formula} is an expression of the form $v \colon A$
   where $v \in \mathit{NC}$ and $A$ is a formula.
\end{definition}

\begin{definition}
   A \emph{tableau} for a set of signed formulas $\Delta$ is a
   downward rooted tree of signed formulas where each one is either an
   element of $\Delta$ or results from a signed formula in the branch
   above it by a branch expansion rule. A tableau is \emph{closed} if
   every branch contains, for some formula $A$, the signed formulas $v
   \colon A$ for all $v \in \mathit{NC}$.
\end{definition}

\FOR \op=1 \TO \NoOps \DO
   {\noindent
    The branch expansion rules for connective $\Opname\op$ are given by
    \begin{rules}
       \FOR \tv=1 \TO \NoTVs \DO%
       \ifthenelse{\equal{\TabOpconcl{\op}{\tv}}{}}{}{%
          \setrule{%
            $\begin{array}[t]{@{}c@{}}
              \TabOpprem{\op}{\tv} 
              \\ \hline
              \TabOpconcl{\op}{\tv} 
               \end{array}$}%
      }%
       \ENDFOR
    \end{rules}%
   }
\ENDFOR
\FOR \qu=1 \TO \NoQus \DO
   {\noindent
    The branch expansion rules rules for quantifier $\Quname\qu$ are given by
    \begin{rules}
       \FOR \tv=1 \TO \NoTVs \DO
          \setrule{$\begin{array}[t]{@{}c@{}}
              \TabQuprem{\qu}{\tv} \\ \hline
              \TabQuconcl{\qu}{\tv}
            \end{array}$}%
       \ENDFOR
    \end{rules}%
   }
\ENDFOR

\begin{definition}
   An interpretation $\I$ \emph{satisfies} a signed formula $v\colon
   A$ iff $\val_\I(A) \neq v$. A set of signed formulas is
   \emph{satisfiable} if some interpretation~$\I$ satisfies all signed
   formulas in it.
\end{definition}

\begin{theorem}
   A set of signed formulas is unsatisfiable iff it has a closed tableau.
\end{theorem}

\begin{proof}
   Apply Theorems 4.14 and 4.21 of \citet{Hahnle1993}; interpreting $v
   \colon A$ as $\textsf{S}\ A$ where $\mathsf{S} = \mathit{NC} \setminus
   \{v\}$.
\end{proof}

\begin{corollary}
   In \NC, $\Delta \vDash A$ iff $\{v \colon B \mid v \in
   \mathit{NC}\setminus\mathit{NC}^+, B \in \Delta\} \cup \{v \colon A
   \mid v \in \mathit{NC}^+\}$ has a closed tableau.
\end{corollary}

\end{document}